\documentclass[12pt,letterpaper]{amsart}

\usepackage{palatino, euler, epic, eepic, xy, microtype, stmaryrd, epsfig, epstopdf,xypic}
\usepackage{graphicx,float}

\xyoption{all}

\usepackage{enumerate}
\usepackage[a4paper,bindingoffset=0.2in,left=0.8in,right=1.2in,top=1in,bottom=1in,footskip=.25in,marginparwidth=0.8in]{geometry}
\usepackage{float}
\usepackage{caption}
\usepackage{amsmath, amssymb, amsfonts, amsthm, accents, xcolor,enumerate, esint}
\usepackage{mathtools}
\usepackage{tikz-cd}
\usepackage{empheq}
\usepackage{adjustbox}

\newcommand{\snb}[1]{}

\usepackage{amsmath, amssymb, amscd, verbatim, xspace,amsthm}
\usepackage{latexsym, epsfig, color}

\newcommand{\A}{\ensuremath{{\mathbb{A}}}}

\newcommand{\G}{\ensuremath{{\mathbb{G}}}}

\newcommand\tensor{\otimes}
\newcommand\isom{\cong}

\newcommand\bq{\begin{equation}}
\newcommand\eq{\end{equation}}

\newtheorem{proposition}{Proposition}[section]
\newtheorem{theorem}[proposition]{Theorem}
\newtheorem{corollary}[proposition]{Corollary}
\newtheorem{example}[proposition]{Example}

\newtheorem{lemma}[proposition]{Lemma}

\theoremstyle{definition}
\newtheorem{definition}[proposition]{Definition}

\theoremstyle{remark}
\newtheorem{remark}[proposition]{Remark}
\usepackage{url}

\numberwithin{equation}{section}

\newcommand{\cut}[1]{}

\newcommand\hidden[1]{}

\newcommand{\cF}{\mathcal{F}}

\newcommand{\FF}{\mathbb{F}}

\newcommand{\PP}{\mathbb{P}}
\newcommand{\QQ}{\mathbb{Q}}
\newcommand{\RR}{\mathbb{R}}

\newcommand{\ZZ}{{\mathbb{Z}}}                                        %
\newcommand{\cO}{{\mathcal O}}                                        %
\newcommand{\Bl}{\operatorname{Bl}}                                   
\newcommand{\Cl}{\operatorname{Cl}}                                   
\newcommand\oo{\mathcal{O}}                                           %
\newcommand\fp{\overline{\mathbb{F}}_p}                                    %

\usepackage{caption}                                                  %
\usepackage{subcaption}                                               %
\usepackage{tikz-cd}                                                  %
\usetikzlibrary{calc}                                                 %
\usetikzlibrary{positioning}                                          %
\usetikzlibrary{decorations.pathreplacing}                            %
\usetikzlibrary{angles}                                               %
\usetikzlibrary{quotes}                                               %
\usetikzlibrary{shapes.misc}                                          %
\usetikzlibrary{decorations.text}                                     %
\usetikzlibrary{arrows}                                               %
\usetikzlibrary{spy}
\usetikzlibrary{intersections}
\usetikzlibrary{through}
\usetikzlibrary{backgrounds}

\title{Constructing non-Mori Dream Spaces from negative curves}

\author{Javier Gonz\'alez Anaya, Jos\'e Luis Gonz\'alez and Kalle Karu}
\address{
J. Gonz\'alez-Anaya, Department of Mathematics, University of British Columbia, 
  Vancouver, BC V6T1Z2, Canada.  \newline \indent
J.L. Gonz\'alez,  Department of Mathematics, University of California, Riverside,
  Riverside, CA 92521, United States.  \newline \indent
K. Karu,
Department of Mathematics, University of British Columbia, 
  Vancouver, BC V6T1Z2, Canada.} 
\email{jga@math.ubc.ca, jose.gonzalez@ucr.edu, karu@math.ubc.ca}
\thanks{The first author was supported by CONACyT scholarship 410172.}
\thanks{The second author was supported by the UCR Academic Senate.}
\thanks{The third author was supported by a NSERC Discovery grant.}

\begin{document}
\begin{abstract}
We study blowups of weighted projective planes at a general point, and more generally blowups of toric surfaces of Picard number one. Based on the positive characteristic methods of Kurano and Nishida,  we give a general method for constructing examples of Mori Dream Spaces and non-Mori Dream Spaces among such blowups. Compared to previous constructions, this method uses the geometric properties of the varieties and applies to a number of cases. We use it to fully classify the examples coming from two families of negative curves.  

\end{abstract}
\maketitle
\setcounter{tocdepth}{1} 




\section{Introduction}

We work over an algebraically closed field $k$ of characteristic zero.

Recall that a variety $X$ is called a Mori Dream Space (MDS) if its Cox ring is a finitely generated $k$-algebra. In this article we study blowups of weighted projective planes $\PP(a,b,c)$ at a general point $t_0$,
\[ X=\Bl_{t_0} \PP(a,b,c).\]
The problem of determining all triples $(a,b,c)$ for which $X$ is a MDS is largely open.  
Our goal here is to construct new examples of MDS and non-MDS among such $X$, generalizing the methods in \cite{KuranoNishida, GGK}.

Cox rings of general varieties and the MDS  property were first defined by Hu and Keel \cite{HuKeel}.  However, the same problem for the blowups of weighted projective planes has a long history in commutative algebra where these Cox rings are studied because their finite generation is equivalent to that of the symbolic Rees algebra of the corresponding monomial ideal (see for example the work by Cowsik \cite{Cowsik}, Huneke \cite{Huneke}, Srinivasan \cite{Srinivasan}, Cutkosky \cite{Cutkosky}).  Goto, Nishida and Watanabe \cite{GNW} constructed the first examples of non-MDS among such $X$. More recently, Castravet and Tevelev \cite{CastravetTevelev} used one example by Goto, Nishida and Watanabe to show that the moduli spaces $\overline{M}_{0,n}$ are not MDS for $n$ large. These results were later strengthened and generalized in \cite{GK,  He,  HausenKeicherLaface}.

Finite generation of the Cox ring of $X$ is closely related to the existence of negative curves in $X$. Here we use the term ``negative curve'' to mean ``an irreducible curve of negative self-intersection, different from the exceptional curve of the blowup''. By a result of Cutkosky \cite{Cutkosky}, a variety $X$ as above is a MDS if and only if it contains a negative curve $C$ (and more generally, an irreducible curve $C$ with nonpositive self-intersection, \cite{GGK}) and a curve $D$ disjoint from $C$. Such a curve $C$ then generates a boundary ray of the effective cone of $X$ and $D$ generates a boundary ray of the nef cone of $X$.

All examples of non-MDS $X$ mentioned above contain a negative curve that vanishes to order $m=1$ at the point $t_0$. Kurano and Nishida  \cite{KuranoNishida} gave the first examples of non-MDS where $m=2$. In \cite{GGK} we  used the characteristic $p$ methods of Kurano and Nishida to generalize these examples to arbitrary $m>0$, giving an infinite family of MDS and non-MDS $X$. 
In this article we find a second family of examples and give a uniform proof of the MDS and non-MDS properties that works in both cases and expands the results in \cite{GGK}. The proof uses the geometry of the variety $X$ and does not rely on explicit computations as in \cite{KuranoNishida, GGK}. The proof is likely to apply for other families. 

Our examples will all contain negative curves. We prove that $X$ is not a MDS by showing that X does not contain any curves $D$ disjoint from $C$.

To construct examples of $X$ that contain a negative curve, we start with an irreducible curve $C^0$ in the 2-torus $T\isom \G_m^2$, vanishing to order $m$ at $t_0=(1,1)$. We compactify $T$ to a toric variety $X_\Delta$ by choosing a triangle $\Delta$ in $\RR^2$ that contains the Newton polytope of $C^0$. If the triangle has area less than $\frac{m^2}{2}$, then the strict transform of $C^0$ is a negative curve in 
\[ X = \Bl_{t_0} X_\Delta.\]

\begin{remark}
Toric varieties defined by triangles $\Delta$ include all weighted projective planes, but are in general isomorphic to quotients of weighted projective planes by a finite subgroup of the torus.
\end{remark}

We consider two families of curves $C^0$ and for each such curve we give infinitely many varieties $X$ that are MDS and infinitely many that are not MDS. 
The two families of curves are indexed by integers $m>0$ and can be described by their Newton polytopes:
\begin{enumerate}
	\item Let $\Delta_1^0(m)$ be the triangle with vertices $(-1,-1), (m-1,0), (0,m-1)$. (In the case $m=1$ the triangle degenerates to an interval.)
	\item Let $\Delta_2^0(m)$ be the triangle with vertices: $(-1,-1), (m-1,0),\left(\frac{2m-3}{4},\frac{2m-1}{2}\right)$.
\end{enumerate}    
We show in Proposition~\ref{prop:family} that each of these triangles contains the Newton polytope of an irreducible curve $C^0$ that vanishes to order $m$ at $t_0$. 

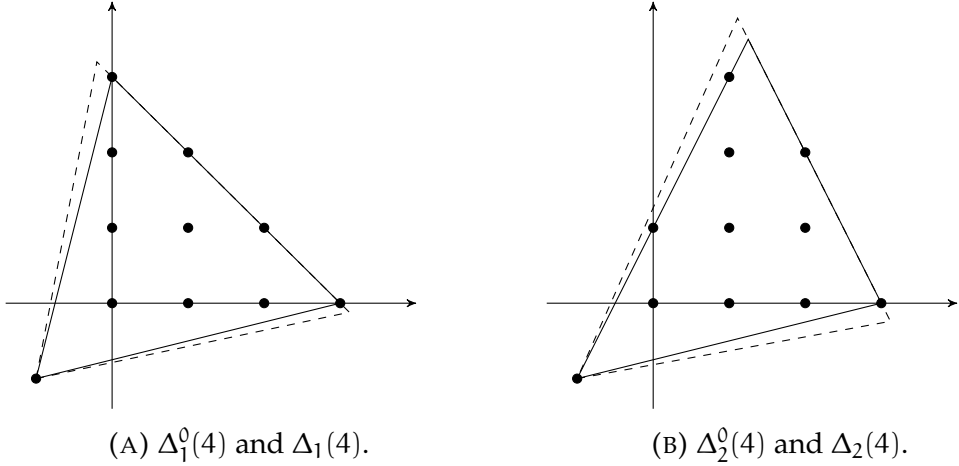
\begin{figure}[htb]
\centering 
  \begin{subfigure}[b]{0.4\linewidth}
    \begin{tikzpicture}[
    scale=1,
    axis/.style={ ->, >=stealth'},
    important line/.style={very thick},
    ]

    \draw[axis] (-1.4,0) -- (4,0) node(xline)[right] {};    
    \draw[axis] (0,-1.4) -- (0,4) node(yline)[above] {};
    \foreach \Point in {
      (-1,-1),
      (0,0),(1,0),(2,0),(3,0),
      (0,1),(1,1),(2,1),
      (0,2),(1,2),
      (0,3)
    }
    \draw[fill=black] \Point circle (0.06);

    \coordinate (A) at (-1,-1);
    \coordinate (B) at (3,0);
    \coordinate (C) at (0,3);
    \coordinate (BB) at (3.12,-0.12);
    \coordinate (CC) at (-0.2,3.2);


    \draw (A) -- (B) -- (C) -- cycle;
    \draw[dashed] (A) -- (BB) -- (CC) -- cycle;

\end{tikzpicture}
    \caption{$\Delta_1^0(4)$ and $\Delta_1(4)$.}  
  \end{subfigure}
\qquad
  \begin{subfigure}[b]{0.4\linewidth}
    \begin{tikzpicture}[
    scale=1,
    axis/.style={ ->, >=stealth'},
    important line/.style={very thick},
    ]

    \draw[axis] (-1.4,0) -- (4,0) node(xline)[right] {};    
    \draw[axis] (0,-1.4) -- (0,4) node(yline)[above] {};
    \foreach \Point in {
      (-1,-1),
      (0,0),(1,0),(2,0),(3,0),
      (0,1),(1,1),(2,1),
      (1,2),(2,2),
      (1,3)
    }
    \draw[fill=black] \Point circle (0.06);

    \coordinate (A) at (-1,-1);
    \coordinate (B) at (3,0);
    \coordinate (C) at (5/4,7/2);
    \coordinate (BB) at (3.12,-0.24);
    \coordinate (CC) at (5/4-0.14,7/2+0.28);


    \draw (A) -- (B) -- (C) -- cycle;
    \draw[dashed] (A) -- (BB) -- (CC) -- cycle;
  
\end{tikzpicture}
    \caption{$\Delta_2^0(4)$ and $\Delta_2(4)$.}  
  \end{subfigure}
\caption{Triangles $\Delta_i^0(4)$ with solid edges and $\Delta_i(4)$ with dashed edges for $i=1,2$.}
\label{fig-both-triangles}
\end{figure}

The toric varieties $X_\Delta$ are constructed by choosing a slightly larger triangle containing $\Delta_1^0$ or $\Delta_2^0$ with vertices:
\begin{enumerate}
	\item $(-1,-1), (m-1,0)+\alpha(1,-1), (0,m-1)+\beta(-1,1)$;
	\item $(-1,-1), (m-1,0)+\alpha(1,-2), \left(\frac{2m-3}{4},\frac{2m-1}{2}\right)+\beta(-1,2)$;
\end{enumerate}
for $\alpha,\beta\geq 0$, see Figure \ref{fig-both-triangles}. We denote these latter triangles by $\Delta_i(m)$ with $i=1,2$. 
The triangles $\Delta_i(m)$ depend on $\alpha$ and $\beta$, which we omit from notation. We will often make a statement for all $m>0$ and also drop $m$ from the notation.

The curves $C$ are negative  in $X$ if $\alpha$ and $\beta$ satisfy:
\begin{enumerate}
	\item $\alpha+\beta< \frac{1}{m+1}$;
	\item $\alpha+\beta< \frac{1}{4(2m+1)}$.
\end{enumerate}

Our main result describes which of the varieties $X$ are MDS.

\begin{theorem} \label{main-thm}
Let $X_\Delta$ with $\Delta=\Delta_i$ for $i=1,2$ be the toric variety defined by one of the two types of triangles as above. Assume that $\alpha$ and $\beta$ are such that $C$ is a negative curve in $X=\Bl_{t_0} X_\Delta$. Then, $X$ is not a MDS if and only if in families (1) and (2), respectively,
\begin{enumerate}
	\item  $\alpha>0$ and $\beta > \frac{1}{m+2}$, or $\beta>0$ and $\alpha > \frac{1}{m+2}$;
	\item  $\alpha>0$ and $\beta >0$.
        \end{enumerate}       
\end{theorem}

The first family of curves $C^0$ is the one considered in \cite{GGK}. However, the theorem here is stronger, enlarging the set of  non-MDS and giving an if and only if statement.

In Theorem~\ref{main-thm} we did not consider the case where $C$ has self-intersection number $0$ (which is the case if $\alpha+\beta=\frac{1}{m+1}$ in the first family, and $\alpha+\beta=\frac{1}{4(2m+1)}$ in the second family). Our proof does not give much information about such $X$.

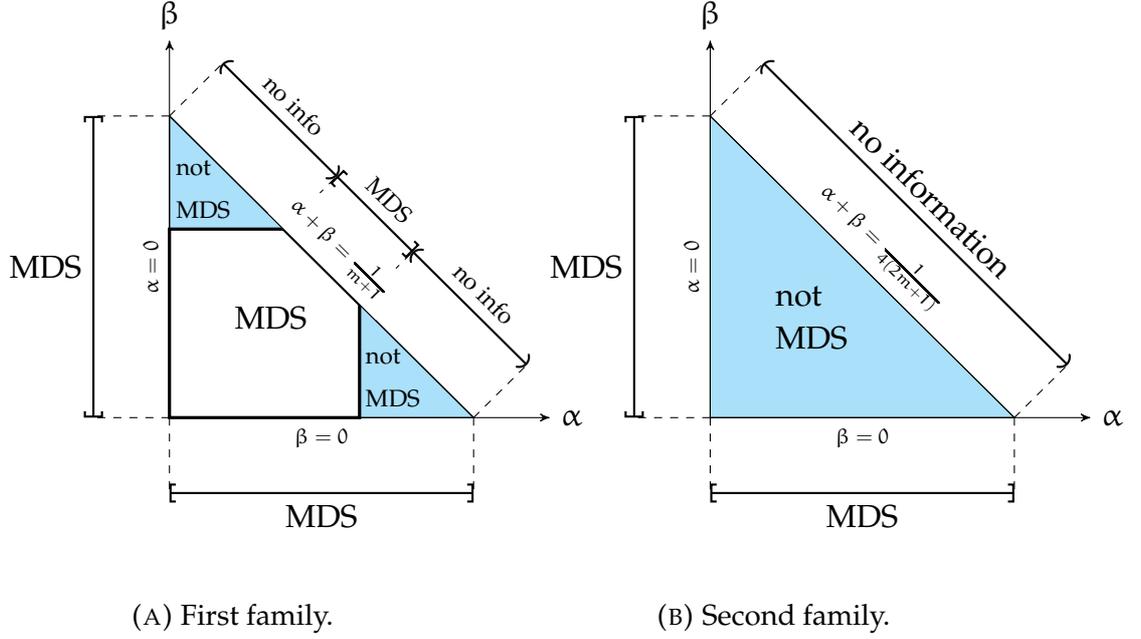
\begin{figure}[htb]
\centering 
  \begin{subfigure}[b]{0.4\linewidth}
    \begin{tikzpicture}[
    scale=1,
    axis/.style={ ->, >=stealth'},
    important line/.style={},
    ]

    \draw[axis] (0,0) -- (5,0) node(xline)[right] {};    
    \draw[axis] (0,0) -- (0,5) node(yline)[above] {};

    \coordinate (A) at (4,0);
    \coordinate (B) at (0,4);
    \coordinate (C) at (0,0);

    \draw (A) -- (B);

    \coordinate (A1) at (2.5,0);
    \coordinate (A2) at (2.5,1.5);
    \coordinate (B1) at (0,2.5);
    \coordinate (B2) at (1.5,2.5);

    \draw[fill = cyan!30] (A1) -- (A2) -- (A) -- cycle;
    \draw[fill = cyan!30] (B1) -- (B2) -- (B) -- cycle;

    \draw (5,0) circle (0) node[right] {$\alpha$};
    \draw (0,5) circle (0) node[above] {$\beta$};
    \node at (barycentric cs:A=1,B=1,C=1) {MDS};
    \node at (3.44,0) [above left,align=left]{\scriptsize not \\ \scriptsize MDS};
    \node at (-0.06,2.5) [above right,align=left]{\scriptsize not \\ \scriptsize MDS};

    \draw[very thick] (0,0) -- (A1) -- (A2) -- (B2) -- (B1) -- cycle;
    
    \begin{scope}[shift={(0,2.5)}]
      \draw[dashed] (0,1.5) -- (-1,1.5);
    \end{scope}

    \begin{scope}[shift={(2.5,0)}]
      \draw[dashed] (1.5,0) -- (1.5,-1);
    \end{scope}

    \begin{scope}[shift={(A)}]
      \draw[dashed] (0,0) -- (0.7,0.7);
    \end{scope}

    \begin{scope}[shift={(B)}]
      \draw[dashed] (0,0) -- (0.7,0.7);
    \end{scope}
    \begin{scope}[shift={(B2)}]
      \draw[dashed] (0,0) -- (0.7,0.7);
    \end{scope}
    \begin{scope}[shift={(A2)}]
      \draw[dashed] (0,0) -- (0.7,0.7);
    \end{scope}

    \draw[dashed] (0,0) -- (-1,0);
    \draw[dashed] (0,0) -- (0,-1);

    \draw[{(-)},thick] (1.5+0.7,2.5+0.7) -- (0.7,4.7) node[midway, above,sloped] {\scriptsize no info};
    \draw[{[-]},thick] (1.5+0.7,2.5+0.7) -- (2.5+0.7,1.5+0.7) node[midway, above,sloped] {\scriptsize MDS};	
    \draw[{(-)},thick] (2.5+0.7,1.5+0.7) -- (4.7,0.7) node[midway, above,sloped] {\scriptsize no info};

    \draw[{[-]},thick] (-1,0) -- (-1,4) node[midway,left] {MDS};
    \draw[{[-]},thick] (0,-1) -- (4,-1) node[midway,below] {MDS};

    \draw (A) -- (B) node[rectangle, fill, fill=white, midway, above, sloped] {\tiny $\alpha+\beta=\frac{1}{m+1}$};
    \draw (C) -- (B) node[midway,above,sloped] {\tiny $\alpha=0$};
    \draw (C) -- (A) node[midway,below,sloped] {\tiny $\beta=0$};

\end{tikzpicture}
    \caption{First family.}  
  \end{subfigure}
\qquad
  \begin{subfigure}[b]{0.4\linewidth}
    \begin{tikzpicture}[
    scale=1,
    axis/.style={ ->, >=stealth'},
    important line/.style={},
    ]

    \draw[axis] (0,0) -- (5,0) node(xline)[right] {};    
    \draw[axis] (0,0) -- (0,5) node(yline)[above] {};

    \coordinate (A) at (4,0);
    \coordinate (B) at (0,4);
    \coordinate (C) at (0,0);

    \draw[fill=cyan!30] (A) -- (B) -- (C) -- cycle;

    \draw (5,0) circle (0) node[right] {$\alpha$};
    \draw (0,5) circle (0) node[above] {$\beta$};

    \node at (barycentric cs:A=1,B=1,C=1) [align=left]{not \\ MDS};
    
    \begin{scope}[shift={(0,2.5)}]
      \draw[dashed] (0,1.5) -- (-1,1.5);
    \end{scope}

    \begin{scope}[shift={(2.5,0)}]
      \draw[dashed] (1.5,0) -- (1.5,-1);
    \end{scope}

    \begin{scope}[shift={(A)}]
      \draw[dashed] (0,0) -- (0.7,0.7);
    \end{scope}

    \begin{scope}[shift={(B)}]
      \draw[dashed] (0,0) -- (0.7,0.7);
    \end{scope}

    \draw[dashed] (0,0) -- (-1,0);
    \draw[dashed] (0,0) -- (0,-1);

    \draw[{(-)},thick] (4.7,0.7) -- (0.7,4.7) node[midway, above,sloped] {no information};
    \draw[{[-]},thick] (-1,4) -- (-1,0) node[midway,left]{MDS};
    \draw[{[-]},thick] (0,-1) -- (4,-1) node[midway,below]{MDS};

    \draw (A) -- (B) node[midway, above, sloped] {\tiny $\alpha+\beta=\frac{1}{4(2m+1)}$};
    \draw (C) -- (B) node[midway,above,sloped] {\tiny $\alpha=0$};
    \draw (C) -- (A) node[midway,below,sloped] {\tiny $\beta=0$};

\end{tikzpicture}
    \caption{Second family.}  
  \end{subfigure}
\caption{Visual representation of the results in Theorem \ref{main-thm} in the $(\alpha,\beta)$-plane.}
\label{fig-both-parameters}
\end{figure} 

\begin{example}\label{example}
  We have the following examples for each of the two of families of triangles in Theorem \ref{main-thm}:
  \begin{enumerate}[(1)]
  \item The parameters
    \[
      \alpha=\frac{3}{4m^2+11m+10},\qquad\beta=\frac{4m+5}{(2m+3)^2}
    \]
for the first family define a sequence of spaces for which Theorem 1.3 in \cite{GGK} does not apply. However, they satisfy the inequalities of Theorem \ref{main-thm} (1).

The normal fan of the triangle $\Delta_1$ has rays generated by
    \[
      (-(4m^2+8m+5),4(m+1)),\quad (4m+7,-(m+1)(4m+3)),\quad (1,1).
    \]
    Then, for $m\not\equiv {1\mod 5}$ the associated variety is:
    \[
      X_\Delta = \PP(4m^2+11m+10,(2m+3)^2,16 m^4 + 60 m^3 + 72 m^2 + 15 m - 13),
    \]
    and its blowup is \emph{not} a MDS.

    For $m=2,3,4,5$ we have the following examples:
    \[
      \PP(48, 49, 1041),\quad\PP(79, 81, 3596),\quad\PP(118, 121, 9135),\quad\PP(165, 169, 19362).
    \]
If $m\equiv {1\mod 5}$ then $X_\Delta$ is a quotient of a weighted projective plane by a finite group. For example, for $m=1$, $X_\Delta$ is isomorphic to the quotient of $\PP(1,1,6)$ by a $3$-element group. Its blowup is still a non-MDS.

  \item For the second family the parameters
    \[
      \alpha=\frac{1}{10m+7},\qquad\beta=\frac{2m+1}{16(20m^2+4m+13)}
    \]
    satisfy the inequalities of Theorem \ref{main-thm} (2).

    The normal fan of the triangle $\Delta_2$ has rays generated by
    \[
      (-2(80m^2+16m+53),80m^2+16m+51),\quad (5,-(5m+1)),\quad (2,1).
    \]

    Thus, for every $m$ the associated variety is:
    \[
      X_\Delta = \PP(10m+7,16(20m^2+4m+13),800m^3-80m^2+482m-149).
    \]
    and its blowup is \emph{not} a MDS.

    For $m=1,2,3$ we have the following examples:
    \[
      \PP(17,592,1053),\quad\PP(27,1616,6895),\quad\PP(37,3280,22177).
    \]
  \end{enumerate}
\end{example}

In Theorem~\ref{main-thm} we consider only a subset of toric varieties $X_\Delta$ such that the blowup contains a negative curve from one of the two families. It is not difficult to construct examples of MDS by also varying the right side of the triangles (with slope $-1$ or $-2$) as was done in \cite{GGK}. Indeed, we will prove below that the cases where $\alpha=0$  or $\beta=0$ give rise to MDS. The same proof works for more general triangles.

The two families of curves studied here arise from triangles $\Delta^0(m)$ that contain exactly ${m+1 \choose 2}+ 1 $ lattice points. (This can be seen by counting lattice points in rows.) Since vanishing to order $m$ imposes ${m+1 \choose 2}$ linear conditions, we are guaranteed to find at least one negative curve in $X$. Determining if a general $X$ contains a negative curve is an open problem. Kurano and Matsuoka in \cite{KuranoMatsuoka} give examples of negative curves in triangles with fewer than ${m+1 \choose 2}+ 1 $ lattice points, and examples of $\Bl_{t_0}\PP(a,b,c)$ which are conjectured  to contain no negative curves.


\section{Negative curves}     \label{section.negative.curves}

In this section we construct the two families of irreducible curves $C^0 \subset T$. These are defined by Laurent polynomials that we will call $\xi_m$ throughout the paper. 

We let $K$ denote any field. The varieties $X_\Delta$ and $X=\Bl_{t_0} X_\Delta$ can be defined over $K$. 

We say that a Laurent polynomial is supported in a triangle $\Delta$ if its Newton polytope lies in $\Delta$.

Let us start with an irreducibility criterion for Laurent polynomials.

\begin{lemma} \label{irreducibility.lemma}
   Let $f(x,y) \in K[x^{\pm 1},y^{\pm 1}]$ be a Laurent polynomial supported in a triangle $\Delta$. 
   Assume that $\Delta$ has an edge $E$ whose only integral points are its two endpoints and that the corresponding coefficients of $f(x,y)$ are nonzero.
   Then, $f(x,y)$ is irreducible in $K[x^{\pm 1},y^{\pm 1}]$. If additionally $f(x,y) \in K[x,y]$ is a polynomial not divisible by $x$ or $y$, then $f(x,y)$ is irreducible in $K[x,y]$.
\end{lemma}
  
\begin{proof} 
Suppose that $f=gh$. Then the Newton polytope of $f$ is the Minkowski sum of the Newton polytopes of $g$ and $h$, $P_f= P_g+P_h$. We show that if the sum of any two integral polytopes $P_g+P_h$ lies in the triangle $\Delta$ and contains its edge $E$, then one of the summands has to be a point.

A face of the Minkowski sum is the Minkowski sum of two faces, one from each summand. Hence can write $E = G+H$, where $G$ is a face of $P_g$ and $H$ is a face of $P_h$. Since $E$ is not the Minkowski sum of two nontrivial integral polytopes, one of the faces, say $G$, has to be a vertex of the corresponding polytope $P_g$. If $P_g$ is not a point, let $F$ be any edge of $P_g$ containing the vertex $G$. Then $F+H$ must lie in $\Delta$. However, $F+H$ is a parallelogram with one edge equal to the edge $E$ of $\Delta$. This shows that $P_g$ is the point $G$.

The last statement follows from the first because monomial terms are the only units in $K[x^{\pm 1}, y^{\pm 1}]$. 
\end{proof}

\begin{remark} 
The previous proof also works for Laurent polynomials over any field supported in a higher dimensional simplex where the simplex has one edge $E$ as before contained in the Newton polytope of $f$. 
\end{remark}

The following proposition constructs the equations $\xi_m$ defining the curves $C^0\subset T$.

\begin{proposition}\label{prop:family}
Consider the triangles $\Delta_i^0(m)$ for $i=1,2$ and $m>0$ from Section 1. There exists a polynomial $\xi_m\in\ZZ[x^{\pm 1}, y^{\pm 1}]$ supported in the triangle $\Delta_i^0(m)$ such that for any field $K$ the polynomial $\xi_m$ considered as an element in $K[x^{\pm 1}, y^{\pm 1}]$ satisfies:
\begin{enumerate}
    \item $\xi_m$ has multiplicity exactly $m$ at $t_0=(1,1)$.
  \item The coefficients of the monomials of $\xi_m$ corresponding to the points $(-1,-1)$ and $(m-1,0)$ are nonzero in $K$. In particular, $\xi_m$ is irreducible in $K[x^{\pm 1}, y^{\pm 1}]$.
  \end{enumerate}
\end{proposition}

\begin{proof}
We will show that for each field $K$ there exists a polynomial $\xi_{m,K}$ supported in $\Delta_i^0(m)$ and satisfying the two conditions of the proposition. Such a polynomial is unique up to a constant multiple because it defines a negative curve in $X_K =  \Bl_{t_0} X_{\Delta_i^0(m)}$, where the varieties are defined over $K$. To get an integer polynomial, we construct $\xi_{m,\QQ}$ and clear its denominators so that the gcd of its coefficients is $1$. Let us call this polynomial $\xi_m$. Notice that this $\xi_m$ is nonzero as a polynomial in $K[x^{\pm 1}, y^{\pm 1}]$, and moreover, it defines a (possibly reducible) curve of negative self-intersection in $X_K$. This implies that $\xi_m$ reduces to a constant multiple of $\xi_{m,K}$ in $K[x^{\pm 1}, y^{\pm 1}]$, thus proving the proposition.

Let us now fix a field $K$ and construct $\xi_{m,K}$. The condition of vanishing to order at least $m$ imposes $\binom{m+1}{2}$ linear conditions on the coefficients of the $\xi_m$. Indeed, the condition translates to $\xi(x+1,y+1)$ having no terms of degree less than $m$. (In the case of a Laurent polynomial we need to expand $(x+1)^{-1} = 1-x+x^2- \ldots$, similarly for $(y+1)^{-1}$, and work with power series in $x$ and $y$.) The triangles $\Delta_i^0(m)$ contain exactly $\binom{m+1}{2}+1$ lattice points hence there is at least one nonzero polynomial supported in $\Delta_i^0(m)$ that vanishes to order at least $m$ at $t_0$. Let us check the two conditions of the proposition for this polynomial.

Assume the coefficient $(-1,-1)$ of $\xi_m$ is zero. For the first family this implies  $\xi_m$ is supported in the right triangle with vertices $(0,0)$, $(m-1,0)$ and $(0,m-1)$. The toric variety associated to this triangle is $\PP^2$. However, the blowup of $\PP^2$ does not contain negative curves (it contains a curve of self-intersection number zero defined by the polynomial $1-y$). Thus, the coefficient cannot be zero.
Similarly, for the second family we get an isosceles triangle with equal height and width. The polynomial $1-y$ defines an irreducible curve with self-intersection number zero, hence there cannot be a curve with negative self-intersection. This proves that the coefficient of $x^{-1} y^{-1}$ in $\xi_m$ is nonzero. 

Now consider the lattice point $(0,m-1)$. In the power series $\xi_m(x+1, y+1)$ the term $x^{m-1}$ comes from two terms: $(x+1)^{-1}(y+1)^{-1}$ and $(x+1)^{m-1}$. We know that the first term contributes a nonzero coefficient to $x^{m-1}$. Hence the second term must contribute the same nonzero coefficient with negative sign. This shows that the coefficient of $x^{m-1}$ in $\xi_m$ is nonzero.

To see that $\xi_{m,K}$ vanishes exactly to order $m$ at $t_0$, consider the term with $x^m$ in the power series $\xi_m(x+1, y+1)$. Only $(x+1)^{-1}(y+1)^{-1}$ contributes to this term with nonzero coefficient. This implies that $\xi_m$ does not vanish to order $m+1$ at $t_0$.
\end{proof}

\begin{remark}
  It is possible to explicitly construct polynomials as in Proposition \ref{prop:family}:
  \begin{enumerate}[(1)]
  \item The first family can be constructed via either of the following recurrence relations:
    \begin{enumerate}[(a)]
    \item $\xi_1=1-\frac{1}{xy}$,
    \item $\xi_{m+1}=(x-1)\xi_m+y^{-1}(y-1)^{m+1}$,
    \item $\xi_{m+1}=(y-1)\xi_m+x^{-1}(x-1)^{m+1}$.
    \end{enumerate}
  \item The second family can be constructed via the following recurrence relation:
    \begin{enumerate}[(a)]
    \item $\xi_1=1-\frac{1}{xy}$,
    \item $\xi_2=\frac{1}{xy}-3+x+y$,
    \item $\xi_{m+2}=(x-1)\xi_{m+1}+x(y-1)^2\xi_m$.
    \end{enumerate}
  \end{enumerate}

 The polynomials $\xi_1$ and $\xi_2$ are the same in both families. This is because, even though the triangles $\Delta^0_i(m)$ are different for $m=1,2$, the configuration of lattice points in them is the same.

The recurrence formulas for the first family were used in the proofs of \cite{KuranoNishida, GGK}. We will not need the recurrences in this article.   
\end{remark}


\section{Divisors and cohomology}
Let us fix the notation concerning divisors and equivalence classes. We will work over an arbitrary field $K$, but in order to make geometric statements, we will assume that $K$ is algebraically closed.

We consider normal $\QQ$-factorial surfaces $X$ defined over $K$.  The class group $\Cl(X)$ is the group of Weil divisors modulo linear equivalence. 
The vector space $N^1(X)=N_1(X)$ is the real vector space of numerical equivalence classes of Weil divisors (equivalently, curves). We denote by $C\cdot D$ the intersection product between curves. The nef cone of $X$ is the cone in $N_1(X)$ generated by classes of nef divisors. 
Its dual cone (also in $N_1(X)$ via the intersection pairing) is the closure of the cone of effective curves of $X$. 

When $X_\Delta$ is a toric variety defined by a rational triangle $\Delta$, an ample $T$-invariant $\QQ$-Weil divisor $H$ corresponds to a rational triangle $\Delta_H$ with sides parallel to the sides of $\Delta$.
Such $\QQ$-Weil divisor $H$ is Weil if and only if the three lines containing the edges of $\Delta_H$ contain lattice points. 
Two such divisors are linearly equivalent if their triangles differ by an integral translation. 
The divisors have the same numerical equivalence class if their triangles differ by a rational translation. (Thus, the size of the triangle gives the numerical equivalence class of the divisor.)

Given a triangle $\Delta_H$ of a Weil divisor $H$,  the space $H^0(X_\Delta, \cO(H))$ is the set of all Laurent polynomials supported in $\Delta_H$.

Let now $X=\Bl_{t_0} X_\Delta$ and let $\pi:X\to X_\Delta$ be the projection. Then, $N_1(X)$ has dimension $2$ with basis the pullback of an ample class $\pi^*H$ and the class of the exceptional curve $E$. 

\subsection{Some cohomological lemmas}   \label{section.cohomology.lemmas}

\begin{lemma}\label{lemma.toric.vanishing} 
Let $X_\Delta$ be a toric variety defined by a rational triangle $\Delta$, and let $A$ be a Weil divisor on $X_\Delta$. Then,
\[ H^1(X_\Delta,\cO_{X_\Delta}(A)) = 0.\]
\end{lemma}

\begin{proof}
  Since $X_{\Delta}$ has Picard-number one, either $A$ is nef or $-A$ is nef. 
  In either case, we conclude that $ H^1(X_\Delta,\oo_{X_\Delta}(A))=0$ by the Demazure and Batyrev-Borisov vanishing theorems in \cite[Theorem 9.3.5]{ToricVarsBookCoxLS}.
\end{proof}

\begin{proposition}         \label{proposition.direct.images} 
Consider the blowup $\pi: X\rightarrow Y$ of a surface $Y$ at a smooth closed point $t_0$, and the sheaf $\mathcal{F}=\oo_X(\pi^*A - mE)$, where $A$ is a Weil divisor on $Y$ and $E$ is the exceptional divisor. 
Then: 
\begin{enumerate}[(a)] \setlength\itemsep{2mm}

  \item  $\displaystyle
\pi_*\mathcal{F} = \oo_Y(A)\tensor \pi_*\oo_X(-mE)=
\begin{cases}
\oo_Y(A),\qquad \text{if } m \leq 0;\\
\oo_Y(A)\tensor I_{t_0}^m,\quad \text{if } m > 0.
\end{cases}
$           
  
\item ${\displaystyle R^1\pi_*\mathcal{F} = 0, \quad \textnormal{if } m\geq -1.}$     
   
\item  $\displaystyle
H^1(X,\mathcal{F})=
\begin{cases}
H^1(Y,\oo_Y(A)),\quad \text{if } m=-1 \text{ or } m=0;\\
H^1(Y,\oo_Y(A)\tensor I_{t_0}^m),\quad \text{if } m > 0.
\end{cases}
$       
 \end{enumerate}
 Here $I_{t_0}$ is the ideal sheaf of the point $t_0$. 

\end{proposition}

\begin{proof}
Part (c) follows directly from (a) and (b).
To prove (a) and (b) we use that the problem is local in $Y$. 

For (a), consider the map $\phi: \oo_{Y}(A)\otimes \pi_*\oo_{X}(mE)\to \pi_*(\pi^*\oo_{Y}(A)\otimes\oo_{X}(mE))$, induced by the adjunction $\pi^*\dashv\pi_*$ from the natural map $\pi^*(\oo_{Y}(A)\otimes\pi_*\oo_{X}(mE))\to \pi^*\oo_{Y}(A)\otimes\oo_{X}(mE)$. 
The map $\phi$ is an isomorphism over any open subset where either $\pi$ is an isomorphism or $A$ is Cartier (by the projection formula). We can cover $Y$ with two open subsets where one of these cases applies.

In (b), replacing $Y$ with a small affine neighborhood of $t_0$, we may assume that $A=0$ and there exists a fiber square
\[ \xymatrix{
 X  \ar[d]^{\pi} \ar[r]^{\psi} &  \Bl_0 \A^2 \ar[d]^{\rho} \\
  Y  \ar[r]^{\phi} &  \A^2, 
}
\]
where the morphism $\phi$ is \'etale. By \cite[Proposition III.9.3]{Hartshorne} we have
\[ R^1 \pi_*\mathcal{F} = R^1 \pi_* \psi^* \cO_{\Bl_0 \A^2} (-mE) \isom \phi^* R^1 \rho_* \cO_{\Bl_0 \A^2} (-mE).\]
Here we have denoted by $E$ also the exceptional curve in $\Bl_0 \A^2$.
We may thus replace the blowup of $Y$ with the blowup of $\A^2$ at the origin. This last morphism is toric and we can use toric vanishing theorems. For $m\geq 0$, the divisor $-mE$ is nef on $\Bl_0 \A^2$ and hence its higher cohomology vanishes. For $m=-1$, the divisor $E$ on $\Bl_0 \A^2$ can be written as the round-down of a nef $\QQ$-divisor $D$, for example
\[ \cO_{\Bl_0 \A^2} (E) = \cO_{\Bl_0 \A^2} \left( \frac{1}{2}\rho^*D_1 + \frac{1}{2} \rho^*D_2  \right) \]
where $D_1, D_2$ are the toric irreducible divisors on $\A^2$. Now, the $\QQ$-divisor $D$ is nef, hence $\cO_{\Bl_0 \A^2} (D)$ has no higher cohomology \cite[Theorem 9.3.5]{ToricVarsBookCoxLS}.
\end{proof}

\begin{corollary}      \label{h1corollary}
  Let $\pi:X\to X_\Delta$ be the blowup of the toric variety $X_\Delta$, associated to a rational triangle $\Delta$, at the point $t_0=(1,1)$. Consider any toric Weil divisor $A$ in $X_\Delta$ and the sheaf $\cF = \oo_X(\pi^*A - mE)$. 
\begin{enumerate}[(a)] \setlength\itemsep{2mm}

\item If $m=-1$ or $m=0$, then ${\displaystyle H^1(X,\cF)=0}$. 

\item If $m>0$, then ${\displaystyle H^1(X,\cF)=0}$ if and only if the evaluation  map $H^0(\oo_{X_\Delta}(A))\to H^0(\oo_{X_\Delta}(A)\otimes \cO_{X_\Delta}/I_{t_0}^m)$ is surjective.

\item If $m=1$, then ${\displaystyle H^1(X,\cF)=0}$ if and only if $H^0(\oo_{X_\Delta}(A))\neq 0$.

\end{enumerate}
 
\end{corollary}

\begin{proof}
  Part (a) follows from Proposition~\ref{proposition.direct.images}\,(c) and Lemma~\ref{lemma.toric.vanishing}.  

 The vanishing in part (b) is by Proposition~\ref{proposition.direct.images}\,(c)  equivalent to the vanishing of $ H^1(X_\Delta,\oo_X(A)\tensor I_{t_0}^k)$. The conclusion now follows by considering the exact sequence
 \[
H^0(\oo_{X_\Delta}(A)) \rightarrow H^0(\oo_{X_\Delta}(A)\otimes\cO_{X_\Delta}/I_{t_0}^m )   \rightarrow   H^1(\oo_{X_\Delta}(A)\otimes I_{t_0}^m)   \rightarrow   H^1(\oo_{X_\Delta}(A))=0.   \]
 
 For (c), notice that for $m=1$ the surjectivity of the evaluation map in (b) is equivalent to $H^0(\oo_{X_\Delta}(A))\neq 0$. Indeed, $H^0(\oo_{X_\Delta}(A)\otimes \oo_{t_0})$ is a one-dimensional vector space and the image of a section $\chi^u \in H^0(\oo_{X_\Delta}(A))$ is nonzero.
\end{proof}


\section{The method of Kurano and Nishida}
\label{section4}

In this section we prove variants of some results of Cutkosky \cite{Cutkosky} and Kurano and Nishida \cite{KuranoNishida} that we will use in the proof of Theorem~\ref{main-thm}. 
These results were originally proved in the case of weighted projective planes. In \cite{GGK} and in this section they are generalized to the case of toric varieties $X_\Delta$ and more general situations.

\subsection{The Huneke condition and the set $HC_K$}

We fix a rational triangle $\Delta$ defining the pair  $(X_\Delta, H)$, a toric variety and an ample class on it. Let $X=\Bl_{t_0} X_\Delta$. We study these varieties defined over various algebraically closed fields $K$. To emphasize the field, let us call the varieties $X_{\Delta,K}$ and $X_K$. We are mainly interested in the case where either $K=k$ is the base field, or where $K=\overline{\FF}_p$.

We also fix a negative curve $C$ in $X$. We assume that $C$ is defined by a polynomial with integer coefficients, hence we have $C_K\subset X_K$. We assume further that $C_K$ is irreducible for any $K$ and its class is equal to $H-mE$, independent of $K$.  This implies that $C_K\subset X_K$ is a negative curve.     

The following theorem was proved by Cutkosky \cite{Cutkosky} and generalized in \cite{GGK}.

\begin{theorem} \label{thm-cut}
Let $X_K$ be as above, with $C_K\subset X_K$ a negative curve. Then $X_K$ is a MDS if and only if there exists a nonzero effective divisor $D_K\subset X_K$ such that $C_K\cap D_K = \emptyset$. \qed
\end{theorem}

The curves $D_K$ in the theorem should be viewed as effective Weil divisors.
The class of $D_K$ is orthogonal to the class of $C_K$, hence it spans the boundary ray of the nef cone of $X_K$. Thus, the class of $D_K$ is determined up to a positive constant.

When the field $K$ has positive characteristic, then  the existence of  $C_K$ implies the existence of $D_K$ and hence $X_K$ is a MDS, see \cite{Artin62, Cutkosky}. The idea of the characteristic $p$ methods is to study these curves $D_K$ when $K$ has characteristic $p$ to say something about the case of characteristic $0$.

Let us fix a class $[D_0]\in Cl(X)$, for example by fixing an actual divisor $D_0$, such that $C\cdot D_0 = 0$. 
Let $[D_0]$ have the form
\[  [D_0]=  \pi^* H' - m'E, \]
 where $m'>0$ is an integer and $H'$ is given by a rational triangle $\Delta'$. The numerical equivalence class of $D_0$ generates the boundary ray of the nef cone of $X_K$ for any $K$. 
 
Following Kurano and Nishida \cite{KuranoNishida} we define the set $HC_K$ as follows.
 
\begin{definition}
\[ HC_K=\{ l\in \ZZ_{>0} | X_K \text{ contains a divisor } D_K\in |l D_0| \text{ such that } C_K\cap D_K = \emptyset\}.\]
\end{definition}

It follows from Theorem~\ref{thm-cut} that $X_K$ is a MDS if and only if $HC_K\neq \emptyset$. 

The set $HC_K$ is closed under addition (by adding the corresponding curves $D_K$). This implies that it is a sub-semigroup of $\ZZ_{>0}$ and there exist integers $l_0, N$ such that    
\[HC_K \subseteq l_0 \ZZ \quad \text{and} \quad HC_K \cap \ZZ_{>N} = l_0 \ZZ \cap \ZZ_{>N}. \] 

Since $C_K\cdot D_K = 0$ and $C_K$ is irreducible, the condition $C_K\cap D_K = \emptyset$ is equivalent to $C_K \not\subseteq D_K$.  
In the examples below we  fix a point $P$ in  $C_K$ and check that $P\notin D_K$. We choose for $P$ a $T$-fixed point in $X_\Delta$ corresponding to a vertex of $\Delta$. Then, $P\in C_K$ if and only if the vertex does not lie in the Newton polytope of the polynomial defining $C_K$. A similar condition holds for $P\in D_K$. It follows that checking if a fixed $l$ lies in $HC_K$ is a finite dimensional linear algebra problem. We look for a polynomial that vanishes to order $lm'$ at $t_0$. The Newton polytope of the polynomial must lie in $l\Delta'$ and include the vertex corresponding to $P$.

\begin{lemma} \label{lemma.pos.0.char}
A fixed $l$ lies in $HC_k$ if and only if it lies in $HC_{\overline{\FF}_p}$ for all primes $p\gg 0$. 
\end{lemma}

The lemma is proved in \cite{KuranoNishida} and \cite[Lemma 5.1]{GGK}. The proof in \cite{GGK} assumes that $C$ passes through a $T$-fixed point $P$, but this assumption can be easily removed.

\begin{proposition} \label{HCProposition}   
Suppose that $l,l+\mu\in HC_K$ for some $l, \mu \in \mathbb{Z}_{>0}$ 
and that $H^1(X,\oo_X(\mu D_0 - nC))=0$ for some $n \in \mathbb{Z}_{>0}$.
Then, $\mu\in HC_K$.
\end{proposition}

\begin{proof}
Let $\xi \in H^0(\oo_{X}(C))$ define $C$, and let $\zeta \in H^0(\oo_{X}(l D_0))$ define $D$ that gives $l\in HC_K$.   Since  $C\cap D = \emptyset$ we have a short exact sequence 
    \[
    \begin{tikzcd}[column sep=small]
      0\ar[r]& \oo_{X}(\mu D_0-nC)\ar{r}{(\zeta,-\xi^n)}&\oo_{X}((l+\mu) D_0-nC)\oplus\oo_{X}(\mu D_0)\ar{r}{\cdot (\xi^n,\zeta)}& \oo_{X}((l+\mu) D_0)\ar[r]&0.
    \end{tikzcd}
  \]
  Indeed, the exactness on the left and the middle are straightforward, and the exactness on the right is easily verified by restricting separately to the complement of $C$ and the complement of $D$.  

   By the assumption that $H^1(X, \oo_{X}(\mu D_0-nC))=0$ we have a surjective homomorphism 
\[
  H^0(\oo_{X}((l+\mu) D_0-nC))\oplus H^0(\oo_{X}(\mu D_0))\longrightarrow H^0(\oo_{X}((l+\mu) D_0)).
\]
Let $\gamma \in H^0(\oo_{X}((l+\mu) D_0))$ be a section giving $l+\mu\in HC_K$.
Then, $\gamma = f\xi^n + g\zeta$ for some $f\in H^0(\oo_{X}((l+\mu) D_0-nC))$ and some $g\in H^0(\oo_{X}(\mu D_0))$. We claim that $g$ does not vanish at any point of $C$, hence giving $\mu\in HC_K$. Let us check the equivalent condition that $g$ does not vanish along $C$. In the equation $\gamma = f\xi^n + g\zeta$ we know that $\xi$ vanishes along $C$ and $\gamma$ does not, hence $g$ does not vanish along $C$.
\end{proof}

In the two families that we consider the classes of $C$ and $D_0$ will have the form
\begin{equation} \label{eqn-CD} 
[C] = \pi^*H-mE, \qquad [D_0] = \pi^*H'-(im+1)E, \qquad i=1,2.
\end{equation}
Let us specialize to this situation.

\begin{corollary}    \label{HCCorollary}
Let the classes of $C$ and $D_0$ have the form (\ref{eqn-CD}). 
\begin{enumerate}[(a)]  \setlength\itemsep{1.5mm}
\item \label{HCCorollaryA} If $l,l+m\in HC_K$ for some $l,m>0$, then $m\in HC_K$.
\item \label{HCCorollaryB} If $l,l+m-1\in HC_K$ for some $l,m-1>0$, then $m-1\in HC_K$.
\item \label{HCCorollaryC} If $l,l+1\in HC_K$ for some $l>0$ and $H^0(\oo_{X_\Delta}(H'-iH))\neq 0$, then $1\in HC_K$.
\end{enumerate}

\end{corollary}

\begin{proof} 
The claims follow from Proposition~\ref{HCProposition} by choosing appropriate $\mu$ and $n$.

(\ref{HCCorollaryA}) If we let $n=im+1$ and $\mu=m$, then $\mu D_0-nC=\pi^*(\mu H'-nH)$. By Corollary~\ref{h1corollary}\,(a), $H^1(\oo_{X}(\mu D_0-nC))=0$.

(\ref{HCCorollaryB}) If we let $n=i(m-1)+1$ and $\mu=m-1$, then $\mu D_0-nC=\pi^*(\mu H'-nH)+E$.  By Corollary~\ref{h1corollary}\,(a), $H^1(\oo_{X}(\mu D_0-nC))=0$.

(\ref{HCCorollaryC}) If we let $n=i$ and $\mu=1$, then $\mu D_0-nC=\pi^*(\mu H'-nH)-E$.  By Corollary~\ref{h1corollary}\,(c), $H^1(\oo_{X}(\mu D_0-nC))=0$.
\end{proof}

\begin{proposition}    \label{proposition.m.m-1.enough}
Let the classes of $C$ and $D_0$ have the form (\ref{eqn-CD}). 
 Assume that for all $p\gg 0$ there exists $n_p \in \mathbb{Z}_{\geq 0}$ such that $p^{n_p}\in HC_{\fp}$. Then, $HC_k$ is not empty if and only if $m-1\in HC_k$ or, equivalently, $m\in HC_k$. 
\end{proposition}
\begin{proof}
By Lemma~\ref{lemma.pos.0.char}, a fixed $l$ lies in $HC_k$ if and only if $l$ lies in  $HC_{\overline{\FF}_p}$ for all $p \gg 0$.
Suppose that $HC_k$ is not empty and fix $l_0\in HC_k$. 
Then, $l_0\in HC_{\overline\FF_p}$ for $p \gg 0$. 
Since $HC_{\overline\FF_p}$ is a subsemigroup of $\mathbb{Z}_{> 0}$, there exist $l_p, N_p \in \mathbb{Z}_{>0}$ such that $HC_{\overline\FF_p} \subseteq l_p \ZZ$ and $HC_{\overline\FF_p} \cap \mathbb{Z}_{>N_p} = l_p \ZZ \cap \mathbb{Z}_{>N_p}$.
Since $l_0, p^{n_p} \in HC_{\overline\FF_p}$, we deduce that $l_p=1$ for all $p \gg 0$.
Then, by Corollary~\ref{HCCorollary} a) and b), for all $p \gg 0$ we get $m \in HC_{\overline\FF_p}$ and $m-1 \in HC_{\overline\FF_p}$.
Hence, we get $m-1, m \in HC_k$
\end{proof}


\section{Proof of the main theorem}

In this section we prove Theorem~\ref{main-thm} using Proposition~\ref{proposition.m.m-1.enough}.

Let $\Delta = \Delta_i(m)$ be a triangle in one of the two families and $X= \Bl_{t_0} X_\Delta$. The polynomial $\xi_m$ as in Section~\ref{section.negative.curves} defines the negative curve $C$ in $X$ having class
\[ [C] = \pi^*H-mE,\]
where $H$ corresponds to the triangle $\Delta$. Let us choose the class $[D_0]$ of the form 
\[ [D_0] = \pi^*H'-(im+1)E,\]
where $H'$ corresponds to a triangle $\Delta'$ with sides parallel to the sides of $\Delta$. Such a triangle $\Delta'$ is determined by two of its vertices:
  \begin{itemize}
  \item For $i=1$, let $\Delta'$ have two vertices $(m,1)$ and $(0,m+1)$.
  \item For $i=2$, let $\Delta'$ have two vertices $(m,1)$ and $(0,2m+1)$.
  \end{itemize}
Figure~\ref{figure.D} shows the triangles $\Delta'$ for $m=4$.

\begin{lemma}
$C \cdot D_0 = 0. $
\end{lemma}

\begin{proof} 
The area $A$ of the triangle $\Delta$ in the two families is
\[ A_1 = \frac{1}{2}(m+1)(m-1+\alpha+\beta), \qquad A_2 = \frac{1}{2} (2m+1)\left(\frac{2m-1}{4}+\alpha+\beta\right).\]
We get the triangle $\Delta'$ by multiplying  $\Delta$ with a constant $\lambda$ and translating the result. Considering the right edges of the triangles, the constant $\lambda$ can be found to be
\[  \lambda_1=\frac{m}{m-1+\alpha+\beta}, \qquad\lambda_2=\frac{m}{\frac{2m-1}{4}+\alpha+\beta}. \]
Now we can compute the intersection number:
\[ C \cdot D_0 = H\cdot H' -m(im+1) = \lambda H\cdot H - m(im+1) = 2\lambda A - m(im+1) = 0.\]
\end{proof}

    \begin{figure}[H]
    \centering 
    \begin{subfigure}{0.4\linewidth}
      \begin{tikzpicture}[
    scale=1,
    axis/.style={ ->, >=stealth'},
    important line/.style={very thick},
    ]

    \draw[axis] (-1.4,-1) -- (5,-1) node(xline)[right] {};    
    \draw[axis] (0,-1.4) -- (0,5) node(yline)[above] {};
    \foreach \Point in {
      (0,0),(1,0),(2,0),(3,0),(4,0),
      (0,1),(1,1),(2,1),(3,1),
      (0,2),(1,2),(2,2),
      (0,3),(1,3),
      (0,4),
    }
    \draw[fill=black] \Point circle (0.06);
    \foreach \Point in {
      (0,-1),(-1,0),(-1,-1)
    }
    \draw \Point circle (0.08);
    \coordinate (A) at (-4/3,-4/3);
    \coordinate (B) at (4,0);
    \coordinate (C) at (0,4);

    \coordinate (AA) at (-1888/1827+0.09, -11132/9135+0.19);


    \draw [dashed] (A) -- (B) -- (C) -- cycle;
    \draw (AA) -- (B) -- (C) -- cycle;

    \node[rectangle, rounded corners,right] at (B) {\tiny $(m,1)$};
    \node[rectangle, rounded corners,above left] at (C) {\tiny $(0,m+1)$};
\end{tikzpicture}
      \caption{$\Delta'$ in the first family.}  
    \end{subfigure}
    \qquad
    \begin{subfigure}{0.4\linewidth}
      \begin{tikzpicture}[
    scale=0.6,
    axis/.style={ ->, >=stealth'},
    important line/.style={very thick},
    ]

    \draw[axis] (-6,-1) -- (5,-1) node(xline)[right] {};    
    \draw[axis] (0,-3) -- (0,9) node(yline)[above] {};
    \foreach \Point in {
      (-3,0),(-2,0),(-1,0),(0,0),(1,0),(2,0),(3,0),(4,0),
      (-3,1),(-2,1),(-1,1),(0,1),(1,1),(2,1),(3,1),
      (-2,2),(-1,2),(0,2),(1,2),(2,2),(3,2),
      (-2,3),(-1,3),(0,3),(1,3),(2,3),
      (-1,4),(0,4),(1,4),(2,4),
      (-1,5),(0,5),(1,5),
      (0,6),(1,6),
      (0,7),
      (0,8),
      (-4,-1),(-3,-1),(-2,-1),(-1,-1)
    }
    \draw[fill=black] \Point circle (0.06);
    \foreach \Point in {
      (0,-1),
      (-1,6),
      (-2,4),
      (-3,2),
      (-4,0),(-4,-2),
      (-5,-2)
    }
    \draw \Point circle (0.08);

    \coordinate (A) at (-36/7,-16/7);
    \coordinate (B) at (4,0);
    \coordinate (C) at (0,8);

    \coordinate (AA) at (-5.07283+0.5, -2.1602+0.28);


    \draw [dashed] (A) -- (B) -- (C) -- cycle;
    \draw (AA) -- (B) -- (C) -- cycle;
    \node[rectangle, rounded corners,right] at (B) {\tiny $(m,1)$};
    \node[rectangle, rounded corners,above right] at (C) {\tiny $(0,2m+1)$};
    
\end{tikzpicture}
      \caption{$\Delta'$ in the second family.}  
    \end{subfigure}
    \caption{Triangles $\Delta'$ for $m=4$. The dashed triangle corresponds to the values $\alpha=\beta=0$. The solid triangles correspond to $\alpha$ and $\beta$ nonzero.}
    \label{figure.D}
  \end{figure}
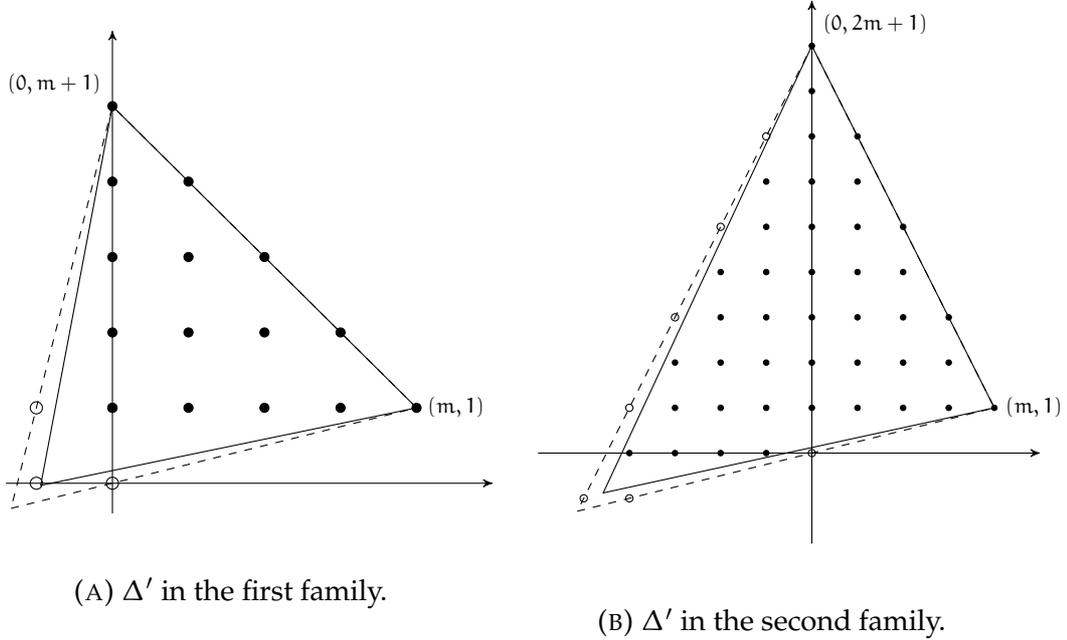  

\begin{lemma}\label{lemma.char.p}
For every prime $p$ there exists an $n_p \geq 0$ such that $p^{n_p}\in HC_{\fp}$.
\end{lemma}
\begin{proof} We show the existence of polynomials $\zeta = \zeta_i$ yielding $p^{l}\in HC_{\fp}$ for all $l\gg 0$.

  Consider the polynomial $(1-y)^{im+1}\in\fp[x,y]$, with Newton polytope containing all the integral points $(0,0),\dots,(0,im+1)$ and vanishing to order $im+1$ at $t_0$.      
  If $\alpha=0$, the polynomial $(1-y)^{mi+1}$ is supported in $\Delta'$ and yields $1\in HC_{\fp}$, and hence $p^l\in HC_{\fp}$ for all $l\geq n_p$.
  Hence, we can now assume $\alpha >0$ for both families and additionally $\beta > 0$ for the first family, by symmetry.
  Thus, with exception of its constant term, the polynomial $((1-y)^{im+1})^{p^l}$ is supported in the triangle $p^l\Delta'$.

  To show the result we will prove that for $l\gg 0$ there exists another polynomial $F$ such that it vanishes to order $p^l(im+1)$ at $t_0$ and is supported in $p^l\Delta'$, except for its constant term which has a nonzero coefficient.  Then, multiplying $F$ by a scalar if necessary, $\zeta=((1-y)^{im+1})^{p^l}+F$ is supported in $p^l\Delta'$ and vanishes to order $p^l(im+1)$ at $t_0$. Further, we construct the polynomial $F$ so that its Newton polytope does not include the top vertex of $p^l\Delta'$. Since the top vertex lies in the Newton polytope of $((1-y)^{im+1})^{p^l}$, it follows that $\zeta$ does not vanish at the $T$-fixed point corresponding to the top vertex. Since $C$ passes through that point, then $\zeta$ cannot vanish on $C$ and hence defines $D$ such that $C\cap D =\emptyset$.

Consider the triangle $\Delta''$ as shown in Figure~\ref{figure.D-subtriangles}. Its right edge lies on the line through the origin. This triangle corresponds to the class of a  Weil divisor $H''$ in $X_\Delta$. 
Since $\Delta'$ is (a translated copy of) the Minkowski sum $\Delta+\Delta''$, we get $H'=H+H''$ in $N_1(X_\Delta)$.

  \begin{figure}
    \centering 
    \begin{subfigure}{0.4\linewidth}
      \begin{tikzpicture}[
    scale=1,
    axis/.style={ ->, >=stealth'},
    important line/.style={very thick},
    ]

    \draw[axis] (-1.4,-1) -- (5,-1) node(xline)[right] {};    
    \draw[axis] (0,-1.4) -- (0,5) node(yline)[above] {};
    \foreach \Point in {
      (1,1),(2,1),(3,1),
      (1,2),(2,2),
      (1,3),
      (0,-1),(1,0),(2,0),(3,0),(4,0)
    }
    \draw[fill=white] \Point circle (0.06);
    \foreach \Point in {
      (1,1),(2,1),(3,1),
      (1,2),(2,2),
      (1,3),
      (0,-1),(1,0),(2,0),(3,0),(4,0)
    }
    \draw[fill=black,opacity = 1] \Point circle (0.06);

    \coordinate (A) at (-4/3,-4/3);
    \coordinate (B) at (4,0);
    \coordinate (C) at (0,4);

    \coordinate (AA) at (-1888/1827+0.09, -11132/9135+0.19);


    \draw [dashed] (A) -- (B) -- (C) -- cycle;
    \draw [name path=AA--B] (AA) -- (B);
    \draw [name path=AA--C] (AA) -- (C);
    \draw (B) -- (C);
    \draw [draw = none, name path=invisible] (-1,0) -- (0,-1);
    \path [name intersections={of=AA--B and invisible,by=TD}];
    \path [name intersections={of=AA--C and invisible,by=TU}];
    \draw (TU) -- (TD) -- (AA) -- cycle;

    \node at (-4/5+0.17,-4/5+0.07) {$\Delta''$};
    \node at (0.8,3.9) {$\Delta'$};
    \node at (2.4,2.4) {$\Delta$};

    \def\shiftedTop#1{
      \draw[draw = none, name path = inv-shifted-top] ($ #1 $) -- ($ (C) + #1 -(AA) $);
    };

    \def\shiftedBottom#1{
      \draw[draw = none, name path = inv-shifted-bottom] ($ #1 $) -- ($ (B) + #1 -(AA) $);
    };
    \def\shiftedTri#1{
      \shiftedTop{#1};
      \shiftedBottom{#1};
      \coordinate (temp) at #1;
      \draw [draw = none, name path=B--C] ((5,-1) -- (C);
      \path [name intersections={of=B--C and inv-shifted-top,by=BShift}];
      \path [name intersections={of=B--C and inv-shifted-bottom,by=CShift}];
      \draw (temp.center) -- (BShift) -- (CShift) -- (temp.center);
    };

    \shiftedTri{(TD)}
    \node[rectangle, rounded corners,right] at (B) {\tiny $(m,1)$};
    \node[rectangle, rounded corners,above left] at (C) {\tiny $(0,m+1)$};
      
\end{tikzpicture}
      \caption{First family.}  
    \end{subfigure}
    \qquad
    \begin{subfigure}{0.4\linewidth}
      \begin{tikzpicture}[
    scale=0.6,
    axis/.style={ ->, >=stealth'},
    important line/.style={very thick},
    ]

    \draw[axis, opacity=0.5] (-6,-1) -- (5,-1) node(xline)[right] {};    
    \draw[axis, opacity = 0.5] (0,-3) -- (0,9) node(yline)[above] {};
    \foreach \Point in {
      (1,0),(2,0),(3,0),(4,0),
      (1,1),(2,1),(3,1),
      (2,2),(3,2),
      (2,3),
      (0,-1)
    }
    \draw[fill=black] \Point circle (0.06);

    \coordinate (A) at (-36/7,-16/7);
    \coordinate (B) at (4,0);
    \coordinate (C) at (0,8);

    \coordinate (AA) at (-5.07283+0.5, -2.1602+0.28);


    \draw [dashed](A) -- (B) -- (C) -- cycle;
    \draw [name path=AA--B] (AA) -- (B);
    \draw [name path=AA--C] (AA) -- (C);
    \draw (B) -- (C);
    \draw [draw = none, name path=invisible] (-3,5) -- (0,-1);
    \path [name intersections={of=AA--B and invisible,by=TD}];
    \path [name intersections={of=AA--C and invisible,by=TU}];
    \draw [fill = white](TU) -- (TD) -- (AA) -- cycle;

    \node at (-2.2,0.3) {$\Delta''$};
    \node at (2,5.4) {$\Delta'$};

    \def\shiftedTop#1{
      \draw[draw = none, name path = inv-shifted-top] ($ #1 $) -- ($ (C) + #1 -(AA) $);
    };

    \def\shiftedBottom#1{
      \draw[draw = none, name path = inv-shifted-bottom] ($ #1 $) -- ($ (B) + #1 -(AA) $);
    };
    \def\shiftedTri#1{
      \shiftedTop{#1};
      \shiftedBottom{#1};
      \coordinate (temp) at #1;
      \draw [draw = none, name path=B--C] ((4.5,-1) -- (C);
      \path [name intersections={of=B--C and inv-shifted-top,by=BShift}];
      \path [name intersections={of=B--C and inv-shifted-bottom,by=CShift}];
      \draw (temp.center) -- (BShift) -- (CShift) -- (temp.center);
    };

    \shiftedTri{(TD)}
        \draw[white, fill=white] (barycentric cs:TD=1,BShift=1,CShift=1) circle (0.3) node {};
    \node at (barycentric cs:TD=1,BShift=1,CShift=1) {$\Delta$};
    
        \node[rectangle, rounded corners,right] at (B) {\tiny $(m,1)$};
    \node[rectangle, rounded corners,above right] at (C) {\tiny $(0,2m+1)$};
\end{tikzpicture}
      \caption{Second family.}  
    \end{subfigure}
    \caption{The triangles $\Delta$, $\Delta'$ and $\Delta''$ for $m=4$. The triangle labelled $\Delta$ is a translate of the actual $\Delta$. The lattice points shown are the monomials in the support of $xy \xi_m$.}
    \label{figure.D-subtriangles}
  \end{figure}
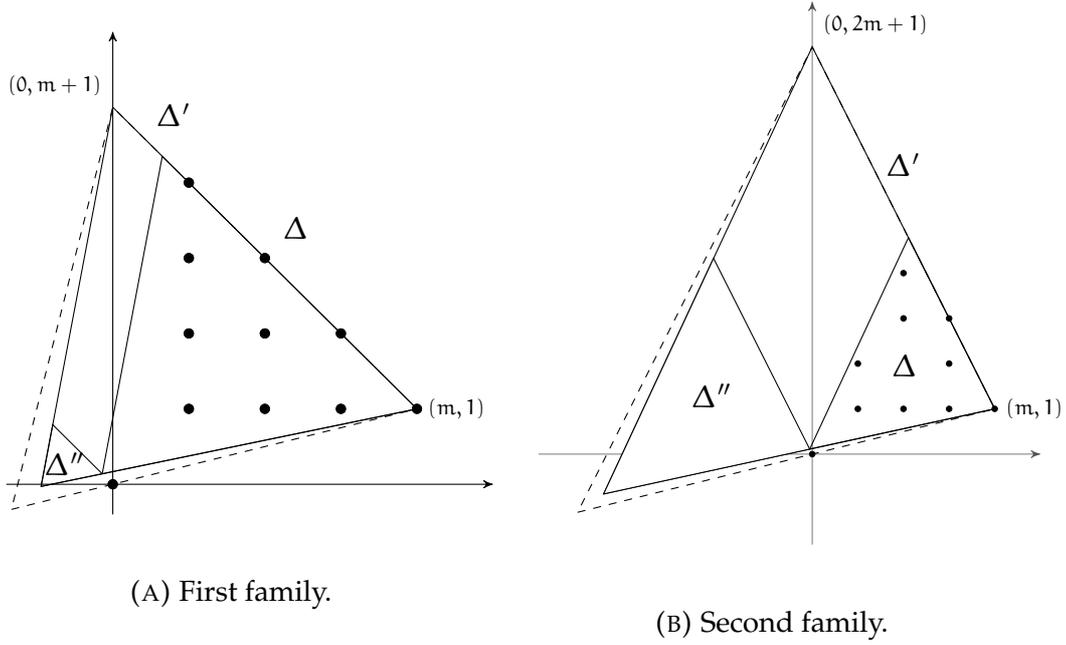

We look for the polynomial $F$  in the form $F= (xy\xi_m)^{p^l}(1+g)$, with $g$ supported in the triangle $p^l \Delta''$. Such $F$ is supported in $\Delta'$, except its nonzero constant term. The polynomial $F$ vanishes to order $p^l(im+1)$ at $t_0$ if $1+g$ vanishes to order $p^l((i-1)m+1)$ at $t_0$. In other words, we are looking for a polynomial $g$ whose restriction to the $p^l((i-1)m+1)$-st order infinitesimal neighborhood of $t_0$ coincides with the function $-1$ on that neighborhood. The existence of such a $g$ follows if we can prove more generally that for any function on the infinitesimal neighborhood there exists a $g$ whose restriction to the neighborhood agrees with the given function. Thus, we want surjectivity of the morphism
\[ H^0(\oo_{X_\Delta}(p^lH''))\to  H^0( \oo_{X_\Delta}(p^lH'') / I_{t_0}^{p^l((i-1)m+1)} ).\]
Let us denote the right hand space by $H^0(\oo_{p^l((i-1)m+1)t_0})$. Then the morphism fits into the long exact sequence
       \[
    \begin{tikzcd}[column sep=small]
      0\ar[r]
      &H^0(\oo_{X_\Delta}(p^lH'')\otimes I_{t_0}^{p^l((i-1)m+1)}) \arrow[r]
      &H^0(\oo_{X_\Delta}(p^lH''))\arrow[r]
      \arrow[d, phantom, ""{coordinate, name=Z}]
      & H^0(\oo_{p^l((i-1)m+1)t_0}) \arrow[dll,
      rounded corners,
      to path={ -- ([xshift=2ex]\tikztostart.east)
        |- (Z) [near end]\tikztonodes
        -| ([xshift=-2ex]\tikztotarget.west)
        -- (\tikztotarget)}] \\
      & H^1(\oo_{X_\Delta}(p^lH'')\otimes I_{t_0}^{p^l((i-1)m+1)}) \arrow[r]
      & 0.
      &
    \end{tikzcd}
  \]
  To guarantee the existence of $g$ it is enough to show that $H^1(\oo_{X_\Delta}(p^lH'')\otimes I_{t_0}^{p^l((i-1)m+1)})=0$ for $l\gg 0$. By Proposition \ref{proposition.direct.images} this translates to showing that $H^1(\oo_X(p^l(H''-nE)))=0$, where $n=(i-1)m+1$. We claim that $H''-nE$ is ample on $X$. Indeed, from $[D_0] = [C] + H''-nE$ we get
  \begin{align*}
    (H''-nE)\cdot C&=(D_0-C)\cdot C=  - C\cdot C > 0,\\
    (H''-nE)\cdot E&= n>0.
  \end{align*}
 Thus, by Kleiman's criterion $H''-nE$ is ample. It follows that its higher cohomology groups vanish for big enough multiples, in particular $H^1(\oo_X(p^l(H''-nE)))=0$ for all $l\gg 0$.
\end{proof}

We will now prove the main theorem. We start by dealing with the values of $\alpha$ and $\beta$ for which $X$ is claimed to be a MDS.

When $\alpha=0$ then the polynomial $\zeta=(1-y)^{mi+1}$ is supported in $\Delta'$ and yields $1\in HC_k$. Indeed, the polynomial $\zeta$ defines a curve $D$ in $X$ that lies in class $[D_0]$, and clearly $C$ is not a component of $D$, hence $C$ and $D$ are disjoint. 

The case $\beta=0$ for the first family follows from the case $\alpha=0$ by symmetry.

In the remaining cases claimed to be MDS, we prove $1\in HC_k$ by constructing a polynomial  $\zeta$ supported in $\Delta'$ and having the top vertex of $\Delta'$ in its support. In this way, the curve $D$ does not pass through the corresponding $T$-fixed point, while the curve $C$ does pass through the same point. This implies that $C$ is not a component of $D$ and hence $C$ and $D$ are disjoint.

Let now $\alpha\leq \frac{1}{m+2}$ and $0 < \beta\leq \frac{1}{m+2}$ in the first family. In this case the point $(-1,0)$ lies in $\Delta'$ (see Figure~\ref{figure.D} (A)). Now $\zeta=y\xi_{m+1}$ is supported in $\Delta'$ and gives $1\in HC_k$. 

Let $\beta=0$ and $\alpha>0$ in the second family. Now in Figure~\ref{figure.D} (B) all lattice points on the dashed left edge lie in the solid triangle $\Delta'$. The polynomials $(1-y)^{2m+1}$ and $x^{-m+1} y \xi_{m+1} \xi_m$ are supported in $\Delta'$, except their nonzero constant terms. A linear combination of these two polynomials defines $D$, giving $1\in HC_k$.

In the remaining cases we need to prove that $X$ is not a MDS. We assume that $\alpha,\beta>0$ and $\beta> \frac{1}{m+2}$ in the first family. By Proposition~\ref{proposition.m.m-1.enough} and  Lemma~\ref{lemma.char.p}, it suffices to prove that $m\notin HC_k$ (the same proof works to show that $m-1 \notin HC_k$). We will assume that $m>1$ and leave it to the reader to check that $1\notin HC_K$ in the case $m=1$. 

Assume by contradiction that $m\in HC_k$, given by a polynomial $\zeta$ that defines the curve $D$ in class $[mD_0]$. The idea of the proof is as follows. The polytope $m\Delta'$ is not the convex hull of its lattice points. We may thus decrease the size of $m\Delta'$ so that it still supports $\zeta$. In fact, we will construct a new triangle $\tilde{\Delta}$ satisfying:
\begin{enumerate}[(a)]
\item The polynomial $1-y$ defines a negative curve $\tilde{C}$ in $\tilde{X} = \Bl_{t_0} X_{\tilde{\Delta}}$.
\item The polynomial $\zeta$ defines a curve $\tilde{D}$ in $\tilde{X}$ such that $\tilde{C}\cdot \tilde{D} <0$.
\end{enumerate}
These properties give a contradiction to the existence of $\zeta$ as follow. Since $\tilde{C}$ is a negative curve in $\tilde{X}$ that intersects $\tilde{D}$ negatively, it follows that $\tilde{C}$ is a component of $\tilde{D}$, in other words, $1-y$ divides $\zeta$. This implies that the left vertex of $m\Delta'$ cannot lie in the support of $\zeta$, hence $D$ passes through the $T$-fixed point corresponding to that vertex. However, $C$ also passes through that point, hence $C$ and $D$ intersect, giving a contradiction.

Consider first the triangle $\Delta^0$ (recall that this is the triangle $\Delta$ with $\alpha=\beta=0$). We modify the slope of its left edge  and ask how large does this slope have to be so that $1-y$ defines a negative curve in the blowup of the resulting toric variety.  If $\overline{\Delta}$  is the new triangle then the correct condition is that the height of $\overline{\Delta}$ is greater than its width. Here the height is measured vertically from the top vertex to the bottom edge and width is measured horizontally from the left vertex to the right vertex. A simple calculation shows that the slope needs to be greater than $m+1+\frac{1}{m}$ for the first family and $2+\frac{1}{m(m+1)}$ for the second family.

    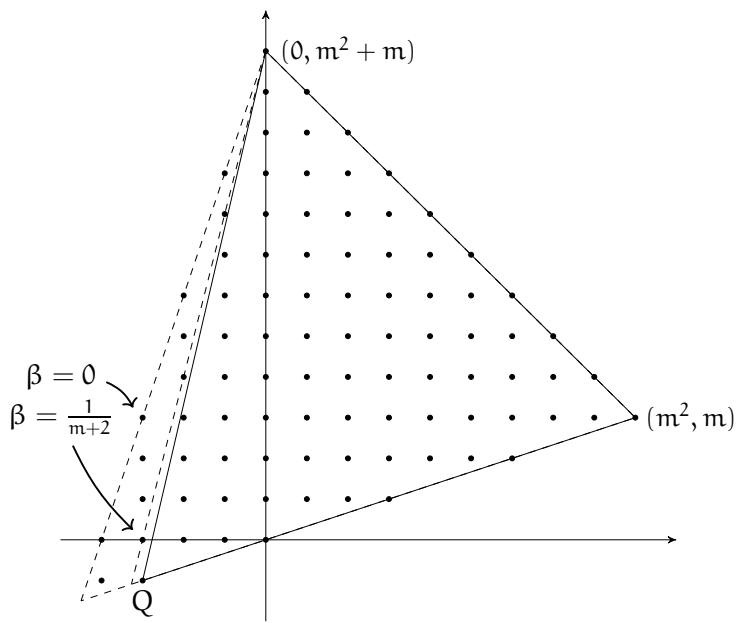
\begin{figure}
    \centering 
      \scalebox{.9}{\begin{tikzpicture}[
    scale=0.6,
    axis/.style={ ->, >=stealth'},
    important line/.style={very thick},
    ]

    \draw[axis] (-5,0) -- (10,0) node(xline)[right] {};    
    \draw[axis] (0,-2) -- (0,13) node(yline)[above] {};
    \foreach \Point in {(-4, -1), (-4, 0), (-3, -1), (-3, 0), (-3, 1), (-3, 2), (-3, 3), (-2, 0), (-2, 1), (-2, 2), (-2, 3), (-2, 4), (-2, 5), (-2, 6), (-1, 0), (-1, 1), (-1, 2), (-1, 3), (-1, 4), (-1, 5), (-1, 6), (-1, 7), (-1, 8), (-1, 9), (0, 0), (0, 1), (0, 2), (0, 3), (0, 4), (0, 5), (0, 6), (0, 7), (0, 8), (0, 9), (0, 10), (0, 11), (0, 12), (1, 1), (1, 2), (1, 3), (1, 4), (1, 5), (1, 6), (1, 7), (1, 8), (1, 9), (1, 10), (1, 11), (2, 1), (2, 2), (2, 3), (2, 4), (2, 5), (2, 6), (2, 7), (2, 8), (2, 9), (2, 10), (3, 1), (3, 2), (3, 3), (3, 4), (3, 5), (3, 6), (3, 7), (3, 8), (3, 9), (4, 2), (4, 3), (4, 4), (4, 5), (4, 6), (4, 7), (4, 8), (5, 2), (5, 3), (5, 4), (5, 5), (5, 6), (5, 7), (6, 2), (6, 3), (6, 4), (6, 5), (6, 6), (7, 3), (7, 4), (7, 5), (8, 3), (8, 4), (9, 3)}
    \draw[fill=black] \Point circle (0.06);
    \coordinate (A) at (-9/2,-3/2);
    \coordinate (B) at (0,12);
    \coordinate (C) at (9,3);
    \coordinate (betamax) at (-36/11,-12/11);
    \coordinate (left) at (-3,-1);
    \coordinate (AA) at (-2,-2);
    \coordinate (a0) at (-5,3);
    \coordinate (a1) at (-3,0);

    \draw[-,dashed] (B) -- (betamax) node[midway,above,sloped] {};

    \draw [draw = none, name path=invisible] (A) -- (C);
    \draw [draw = none, name path=AA--left] (AA) -- (left);
    \draw [name intersections={of=AA--left and invisible,by=ID}];

    \node at (-5,3) [rectangle] (b0) {$\beta=\frac{1}{m+2}$};
    \node at (-3,0) [rectangle] (b1) {};
    \node at (-5,4) [rectangle] (a0) {$\beta=0$};
    \node at (-3,3) [rectangle] (a1) {};
    \draw [->,thick] (a0) to [bend left =20] (a1);
    \def\myshift#1{\raisebox{-2.1ex}}
    \draw [->,thick] (b0) to [bend right =10] (b1);

    

    \draw [dashed] (A) -- (B) -- (C) -- cycle;
    \draw (B) -- (C) -- (left) -- cycle;
    \node[rectangle,rounded corners, below] at (-3,-1) {$Q$};
    \node[rectangle, rounded corners,right] at (C) {\small $(m^2,m)$};
    \node[rectangle, rounded corners,right] at (B) {\small $\,(0,m^2+m)$};

\end{tikzpicture}}
      \caption{Triangle $\tilde\Delta$ for $m=3$ in the first family.}
      \label{figure.E1}
    \end{figure}

    \begin{figure}
    \centering 
    \scalebox{.9}{\begin{tikzpicture}[
    scale=0.6,
    axis/.style={ ->, >=stealth'},
    important line/.style={very thick},
    ]

    \draw[axis] (-7,0) -- (5,0) node(xline)[right] {};    
    \draw[axis] (0,-3) -- (0,11) node(yline)[above] {};
    \foreach \Point in {(-6, -3), (-6, -2), (-5, -2), (-5, -1), (-5, 0), (-4, -2), (-4, -1), (-4, 0), (-4, 1), (-4, 2), (-3, -1), (-3, 0), (-3, 1), (-3, 2), (-3, 3), (-3, 4), (-2, -1), (-2, 0), (-2, 1), (-2, 2), (-2, 3), (-2, 4), (-2, 5), (-2, 6), (-1, 0), (-1, 1), (-1, 2), (-1, 3), (-1, 4), (-1, 5), (-1, 6), (-1, 7), (-1, 8), (0, 0), (0, 1), (0, 2), (0, 3), (0, 4), (0, 5), (0, 6), (0, 7), (0, 8), (0, 9), (0, 10), (1, 1), (1, 2), (1, 3), (1, 4), (1, 5), (1, 6), (1, 7), (1, 8), (2, 1), (2, 2), (2, 3), (2, 4), (2, 5), (2, 6), (3, 2), (3, 3), (3, 4), (4, 2)}
    \draw[fill=black] \Point circle (0.06);
    \coordinate (A) at (-20/3,-10/3);
    \coordinate (B) at (0,10);
    \coordinate (C) at (4,2);
    \coordinate (Q) at (-6,-3);

    \draw [dashed] (A) -- (B) -- (C) -- cycle;
    \draw (Q) -- (B) -- (C) -- cycle;

    \node[rectangle, rounded corners,right] at (C) {\small $(m^2,m)$};
    \node[rectangle, rounded corners,right] at (B) {\small $(0,(2m+1)m)$};
    \node[rectangle, rounded corners, below] at (Q) {$Q$};

\end{tikzpicture}}
      \caption{Triangle $\tilde\Delta$ for $m=2$ in the second family.}
      \label{figure.E2}
    \end{figure}
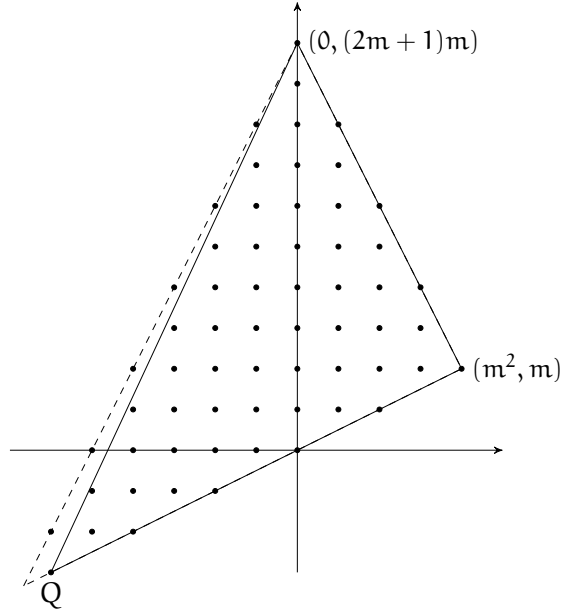
    
    Let us now construct the triangles $\tilde{\Delta}$. We start with the triangle $m\Delta'$. At first step we change the lower edge so that we are in the case $\alpha=0$. This step makes the triangle larger, hence the new triangle still supports $\zeta$. In the second step we pivot the left edge of the triangle about the top vertex. We make the slope as steep as possible so that the triangle still supports $\zeta$ and call the resulting triangle $\tilde{\Delta}$. (See Figures~\ref{figure.E1} and \ref{figure.E2}.) Let us consider the two families separately.
\begin{enumerate}
\item In the first family the assumption $\beta>\frac{1}{m+2}$ implies that the left edge has slope greater than $m+1$. If the slope is exactly $m+1$ then the lattice points that lie on the left edge are
\[ P_i = (0,m^2+m)- i(1,m+1), \qquad i=0,1,\ldots, m.\]
Here $P_0$ is the top vertex of the triangle $m\Delta'$ and $P_m = (-m,0)$ is the lowest point on the left edge. Indeed, this follows from the fact that the left vertex of $m\Delta'$ with $\alpha=0$ and $\beta=\frac{1}{m+2}$ has coordinates
\[ -\left( m+\frac{m}{m^2+m-1}, 1+\frac{1}{m^2+m-1}\right).\]
When we pivot the edge then the first lattice point in $m\Delta'$ that the edge hits is 
\[ Q = P_m - (0,1) = -(m,1).\]
If we let the left edge go through the top vertex and the point $Q$ then it has slope equal to $m+1+\frac{1}{m}$. However, the point $Q$  is not in the support of $\zeta$ because it lies on the lower edge of $m\Delta'$ when $\alpha=0$. Since we assumed $\alpha>0$, we can pivot the left edge a bit more to make its slope greater than $m+1+\frac{1}{m}$. Then $1-y$ defines a negative curve in $\tilde{X}$.

\item In the second family we start with the left edge having slope $2$ and again pivot it about the top vertex of the triangle. When the slope is equal to $2$, then the lattice points on the left edge are 
\[ P_i = (0,(2m+1)m) - i( 1,2), \qquad i= 0,1,\ldots, m^2+m,\]
where $P_0$ is the top vertex of the triangle $m\Delta'$ and $P_{m^2+m} = -(m^2+m, m)$ is the lowest lattice point on the edge. To see that $P_{m^2+m}$ really is the lowest lattice point, note that the left vertex of $m\Delta'$ with $\alpha=\beta=0$ has coordinates
\[ -\left(m^2+m+\frac{m}{2m-1}, m+1+\frac{1}{2m-1}\right).\]
 When we pivot the edge then the first lattice point in $m\Delta'$ that the edge hits is 
\[ Q = P_{m^2+m} - (0,1) = -(m^2+m,m+1).\]
If we let the left edge go through the top vertex and the point $Q$ then it has slope equal to $2+\frac{1}{m(m+1)}$. However, again the point $Q$ is not in the support of $\zeta$ because it lies on the lower edge of $m\Delta'$ when $\alpha=0$. Thus, we can pivot the left edge a bit more to make its slope greater than $2+\frac{1}{m(m+1)}$. Then $1-y$ defines a negative curve in $\tilde{X}$.
 \end{enumerate}

With $\tilde{\Delta}$ defined, let us now check that $\tilde{C} \cdot\tilde{D}<0$. Since $\zeta$ is supported in $\tilde{\Delta}$, it defines the curve $\tilde{D}$ with class 
\[ [\tilde{D}] = \pi^* \tilde{H} - m(im+1)E,\]
where $\tilde{H}$ is the class in $\Cl(X_{\tilde{\Delta}})$ corresponding to the triangle $\tilde{\Delta}$. The curve $\tilde{C}$ has class
\[ [\tilde{C}] = \frac{1}{h} \tilde{H} - E,\]      
where $h$ is the height of $\tilde{\Delta}$, measured vertically from the top vertex. The intersection number is now
\[ \tilde{C} \cdot \tilde{D} = \frac{1}{h} \tilde{H}^2 - m(im+1) = \frac{wh}{h} - m(im+1),\]    
where $w$ is the horizontal width of $\tilde{\Delta}$. Thus, we need to prove that $w< m(im+1)$. However, we have 
\[ w < h = m(im+1).\]
This finishes the proof of the main theorem.


\begin{thebibliography}{10}

\bibitem{Artin62}
Michael Artin.
\newblock Some numerical criteria for contractability of curves on algebraic
  surfaces.
\newblock {\em Amer. J. Math.}, 84:485--496, 1962.

\bibitem{CastravetTevelev}
Ana-Maria Castravet and Jenia Tevelev.
\newblock $\overline{M}_{0,n}$ is not a {M}ori dream space.
\newblock {\em Duke Mathematical Journal}, 164(8):1641--1667, 2015.

\bibitem{Cowsik}
R.~C. Cowsik.
\newblock Symbolic powers and number of defining equations.
\newblock In {\em Algebra and its applications ({N}ew {D}elhi, 1981)},
  volume~91 of {\em Lecture Notes in Pure and Appl. Math.}, pages 13--14.
  Dekker, New York, 1984.

\bibitem{ToricVarsBookCoxLS}
David Cox, John Little, and Henry Schenck.
\newblock {\em Toric varieties}.
\newblock American Mathematical Soc., 2011.

\bibitem{Cutkosky}
Steven~Dale Cutkosky.
\newblock Symbolic algebras of monomial primes.
\newblock {\em J. Reine Angew. Math.}, 416:71--89, 1991.

\bibitem{GK}
Jos\'e~Luis Gonz\'alez and Kalle Karu.
\newblock Some non-finitely generated {C}ox rings.
\newblock {\em Compos. Math.}, 152(5):984--996, 2016.

\bibitem{GGK}
Javier Gonz{\'a}lez-Anaya, Jos{\'e}~Luis Gonz{\'a}lez, and Kalle Karu.
\newblock On a family of negative curves.
\newblock {\em arXiv preprint arXiv:1712.04635}, 2017.

\bibitem{GNW}
Shiro Goto, Koji Nishida, and Keiichi Watanabe.
\newblock Non-{C}ohen-{M}acaulay symbolic blow-ups for space monomial curves
  and counterexamples to {C}owsik's question.
\newblock {\em Proc. Amer. Math. Soc.}, 120(2):383--392, 1994.

\bibitem{Hartshorne}
Robin Hartshorne.
\newblock {\em Algebraic geometry}, volume~52.
\newblock Springer Science \& Business Media, 2013.

\bibitem{HausenKeicherLaface}
J\"{u}rgen Hausen, Simon Keicher, and Antonio Laface.
\newblock On blowing up the weighted projective plane.
\newblock {\em Math. Z.}, 290(3-4):1339--1358, 2018.

\bibitem{He}
Zhuang He.
\newblock New examples and non-examples of {M}ori {D}ream {S}paces when blowing
  up toric surfaces.
\newblock {\em arXiv: 1703.00819}.

\bibitem{HuKeel}
Yi~Hu and Sean Keel.
\newblock Mori dream spaces and {GIT}.
\newblock {\em Michigan Math. J.}, 48:331--348, 2000.
\newblock Dedicated to William Fulton on the occasion of his 60th birthday.

\bibitem{Huneke}
Craig Huneke.
\newblock Hilbert functions and symbolic powers.
\newblock {\em Michigan Math. J.}, 34(2):293--318, 1987.

\bibitem{KuranoMatsuoka}
Kazuhiko Kurano and Naoyuki Matsuoka.
\newblock On finite generation of symbolic {R}ees rings of space monomial
  curves and existence of negative curves.
\newblock {\em J. Algebra}, 322(9):3268--3290, 2009.

\bibitem{KuranoNishida}
Kazuhiko Kurano and Koji Nishida.
\newblock Infinitely generated symbolic {R}ees rings of space monomial curves
  having negative curves.
\newblock {\em arXiv: 1705.09865}.

\bibitem{Srinivasan}
Hema Srinivasan.
\newblock On finite generation of symbolic algebras of monomial primes.
\newblock {\em Comm. Algebra}, 19(9):2557--2564, 1991.

\end{thebibliography}

\end{document}